\documentclass{amsart}

\usepackage{amsmath}
\usepackage{amsthm}
\usepackage{amsfonts}
\usepackage{dsfont}
\usepackage{multirow}

\newtheorem{theorem}{Theorem}[section]
\newtheorem{corollary}[theorem]{Corollary}
\newtheorem{lemma}[theorem]{Lemma}
\newtheorem{proposition}[theorem]{Proposition}
\newtheorem{example}[theorem]{Example}
\theoremstyle{definition}
\newtheorem{definition}[theorem]{Definition}
\theoremstyle{remark}
\newtheorem{remark}[theorem]{Remark}
\numberwithin{equation}{section}

\title[]{Non-Commutative ternary Nambu-Poisson algebras and ternary Hom-Nambu-Poisson algebras}

\author{}
\address{}
\email{}

\author{Hanene Amri}
\address{Universit\'e Badji Mokhtar Annaba, BP 12 d\'epartement de math\'ematiques, facult\'e des sciences, Algeria}
\email{akian296@yahoo.fr}

\author{Abdenacer Makhlouf}
\address{Universit\'{e} de Haute Alsace,  Laboratoire de Math\'{e}matiques, Informatique et Applications,
4, rue des Fr\`eres Lumi\`ere F-68093 Mulhouse, France}
\email{abdenacer.makhlouf@uha.fr}

\thanks{
}


\begin{document}

\begin{abstract}

 The main purpose  of this paper is to study non-commutative ternary Nambu-Poisson algebras  and their Hom-type  version. We provide  construction results dealing with tensor product and direct sums of two (non-commutative) ternary (Hom-)Nambu-Poisson algebras. Moreover, we explore twisting principle of  (non-commutative) ternary Nambu-Poisson algebras along with an algebra
 morphism that lead to construct (non-commutative) ternary Hom-Nambu-Poisson algebras. 
Furthermore,  we  provide examples and 
 a 3-dimensional classification of  non-commutative ternary Nambu-Poisson algebras.

\end{abstract}

\maketitle

\section*{Introduction}\label{sec:intro}
 In the 70's, Nambu proposed a generalized Hamiltonian system based on a ternary product, the Nambu-Poisson bracket, which allows to use more that one hamiltonian \cite{n:generalizedmech}. More recent motivation for ternary brackets appeared in string theory and M-branes, ternary Lie type structure was closely linked to the supersymmetry and gauge symmetry transformations of the world-volume theory of multiple coincident M2-branes and was applied to the study of  Bagger-Lambert theory. Moreover ternary operations appeared in the study of some quarks models.  
In 1996, quantization of Nambu-Poisson brackets were investigated in \cite{GDito},  it was presented in a novel approach of Zariski, this quantization is based on the factorization on $\mathds{R}$ of polynomials of several variables.\\

The algebraic formulation of Nambu mechanics was discussed in \cite{Takhtajan} and    Nambu algebras was studied in \cite{f:nliealgebras} as a natural generalization of a Lie algebra for higher-order algebraic operations. By definition, Nambu algebra of order \emph{n} over a field $\mathbb{K}$ of characteristic zero consists of a vector space \emph{V} over $\mathbb{K}$ together with a $\mathbb{K}$-multilinear skew-symmetric operation $[.,\cdots ,.] :\Lambda^{n}V \rightarrow V$, called the Nambu bracket, that satisfies the following generalization of the Jacobi identity. Namely, for any $x_{1},...,x_{n-1} \in V$ define an adjoint action
$ad(x_{1},...,x_{n-1}):V\rightarrow V$
by
$ad(x_{1},...,x_{n-1})x_{n} = [x_{1},...,x_{n-1}, x_{n}], x_{n} \in V$.

Then the fundamental identity is a condition saying that the adjoint action is a derivation
with respect the Nambu bracket, i.e. for all $x_{1},..., x_{n-1}, y_{1},..., y_{n} \in V$
\begin{equation}\label{NambuIdenAd}
ad(x_{1},..., x_{n-1})[y_{1},..., y_{n}]=
\sum\limits_{k=1}^{n}[y_{1},..., ad(x_{1},..., x_{n-1})y_{k},..., y_{n}].
\end{equation}
 When $n = 2$, the fundamental identity becomes the Jacobi identity and we get a definition of a Lie algebra.

Different aspects of Nambu mechanics, including quantization, deformation and various algebraic constructions for Nambu algebras have recently been studied. Moreover a twisted generalization, called Hom-Nambu algebras, was introduced in \cite{ams:gennambu}. This kind of algebras called Hom-algebras appeared as deformation of algebras of vector fields using $\sigma$-derivations. The first examples concerned $q$-deformations of Witt and Virasoro algebras. Then Hartwig, Larsson and Silvestrov introduced a general framework and studied Hom-Lie algebras \cite{Hartwig}, in which Jacobi identity is twisted by a homomorphism. The corresponding associative algebras, called  Hom-associative algebras was  introduced  in \cite{Makhloufsil}.  Non-commutative Hom-Poisson algebras was discussed in \cite{Yau:Noncomm}. Likewise, $n$-ary algebras of Hom-type was introduced in \cite{ams:gennambu}, see also \cite{ams:MabroukTernary,ams:MabroukRep,ams:ternary,Yau:NaryhomNambu,Yau:HomAssofHomNam}.

We aim in this paper to explore and study non-commutative ternary Nambu-Poisson algebras and their Hom-type version.
The paper includes five Sections. In the first one, we  summarize basic definitions of  (non-commutative) ternary Nambu-Poisson algebras and discuss examples.
In the second Section, we  recall some basics about Hom-algebra structures and  introduce the notion of  (non-commutative) ternary Hom-Nambu-Poisson algebra. Section 3 is dedicated to construction of (non-commutative) ternary Hom-Nambu-Poisson algebras using direct sums and tensor products. In Section 4, we extend twisting principle to ternary Hom-Nambu-Poisson algebras. It is  used to build new structures with a given ternary (Hom-)Nambu-Poisson algebra and  an algebra morphism. This process is used to construct ternary Hom-Nambu-Poisson algebras corresponding to the ternary algebra of polynomials where the bracket is defined by the Jacobian. We provide in the last  section a classification of 3-dimensional ternary Nambu-Poisson algebras and corresponding Hom-Nambu-Poisson algebras using twisting principle.\\

\section{Ternary (Non-commutative) Nambu-Poisson algebra }
In the section we review some basic definitions and fix notations.  In the sequel,  $A$ denotes  a vector space over $\mathbb{K}$, where $\mathbb{K}$ is  an algebraically closed field of characteristic zero.
Let  $\mu:A\times A \rightarrow A$ be a bilinear map,  we denote by  $\mu^{op}:A^{\times 2}\rightarrow A$ the opposite map, i.e., $\mu^{op}=\mu\circ  \tau$ where $\tau: A^{\times 2}\rightarrow A^{\times 2}$ interchanges the two variables.
 A ternary algebra is given by a pair $(A,m)$, where $m$ is a  ternary operation on $A$, that is a trilinear map $m: A\times A\times A \rightarrow  A$, which is denoted sometimes by brackets. 


\begin{definition}
A \emph{ternary Nambu algebra} is a ternary algebra $(A,\{\ , \ , \ \})$  satisfying  the fundamental identity defined as 
\begin{eqnarray}\label{NambuIdentity}
&& \{x_{1},x_{2},\{x_{3},x_{4},x_{5}\}\}=\nonumber \\ &&  \{\{x_{1},x_{2},x_{3}\},x_{4},x_{5}\}+\{x_{3},\{x_{1},x_{2},x_{4}\},x_{5}\}+
\{x_{3},x_{4},\{x_{1},x_{2},x_{5}\}\}
\end{eqnarray}
for all $x_{1},x_{2},x_{3},x_{4},x_{5}\in A$.

This identity is sometimes called   Filippov identity or Nambu identity, and it is equivalent to the identity \eqref{NambuIdenAd} with $n=3$.

A \emph{ternary Nambu-Lie algebra} or $3$-Lie algebra is a ternary Nambu algebra for which the bracket is  skew-symmetric, that is for all $\sigma \in S_{3}$,   where $S_{3}$ is the permutation group, 
    $$[x_{\sigma(1)},x_{\sigma(2)},x_{\sigma(3)}]=Sgn(\sigma)[x_{1},x_{2},x_{3}].$$

Let $A$ and $A'$ be two ternary Nambu algebras (resp. Nambu-Lie algebras). A linear map $f:A\rightarrow A'$ is a \emph{morphism} of a ternary Nambu algebras (resp. ternary Nambu-Lie algebras)  if it satisfies
\begin{center}
$f(\{ x,y,z\}_A)=\{f(x),f(y),f(z)\}_{A'}$.
\end{center}
\end{definition}

\begin{example}
The polynomials of variables $x_{1},x_{2},x_{3}$ with the ternary operation defined by the Jacobian function:
\begin{align}\label{jacobian}
\{f_{1},f_{2},f_{3}\}=\left|
\begin{array}{ccc}
\frac{\partial  f_{1}}{\partial  x_{1}} & \frac{\partial  f_{1}}{\partial  x_{2}} & \frac{\partial  f_{1}}{\partial  x_{3}}\\
\frac{\partial  f_{2}}{\partial  x_{1}} & \frac{\partial  f_{2}}{\partial  x_{2}} & \frac{\partial  f_{2}}{\partial  x_{3}}\\
\frac{\partial  f_{3}}{\partial  x_{1}} & \frac{\partial  f_{3}}{\partial  x_{2}} & \frac{\partial  f_{3}}{\partial  x_{3}}
\end{array}
\right|,
\end{align}
is a ternary Nambu-Lie algebra.
\end{example}
\begin{example}
Let $V=\mathds{R}^{4}$ be the 4-dimensional oriented Euclidian space over $\mathds{R}$. The bracket of 3 vectors $\overrightarrow{x},\overrightarrow{y},\overrightarrow{z}$ is given by

\begin{center}
$$[x,y,z]=\overrightarrow{x}\times\overrightarrow{y}\times\overrightarrow{z}=\left|
\begin{array}{cccc}
x_{1} & y_{1} & z_{1} & e_{1} \\
x_{2} & y_{2} & z_{2} & e_{2} \\
x_{3} & y_{3} & z_{3} & e_{3} \\
x_{4} & y_{4} & z_{4} & e_{4}
\end{array}
\right|,$$
\end{center}
where $\{e_{1},e_{2},e_{3},e_{4}\}$ is a basis of $\mathds{R}^{4}$ and $\overrightarrow{x}=\sum\limits_{i=1}\limits^{3}x_{i}\overrightarrow{e_{i}}$, $\overrightarrow{y}=\sum\limits_{i=1}\limits^{3}y_{i}\overrightarrow{e_{i}}$ and $\overrightarrow{z}=\sum\limits_{i=1}\limits^{3}z_{i}\overrightarrow{e_{i}}$. Then $(V,[.,.,.])$ is a ternary Nambu-Lie algebra.
\end{example}

Now,  we introduce the notion of (non-commutative) ternary Nambu-Poisson algebra.

\begin{definition}
A \emph{non-commutative ternary Nambu-Poisson algebra} is a triple $(A,\mu,\{.,.,.\})$ consisting of  a $\mathbb{K}$-vector space $A$, a bilinear map $\mu:A\times A\rightarrow A$ and a trilinear map $\{.,.,.\}: A\otimes A\otimes A\rightarrow A$ such that
\begin{enumerate}
  \item $(A,\mu)$ is a binary associative algebra,
  \item $(A,\{.,.,.\})$ is a ternary Nambu-Lie algebra,
 \item the following Leibniz rule
 \begin{align*}
\{x_{1},x_{2},\mu(x_{3},x_{4})\}=\mu(x_{3},\{x_{1},x_{2},x_{4}\})+\mu(\{x_{1},x_{2},x_{3}\},x_{4})
\end{align*}
\end{enumerate}
holds for  all $x_{1},x_{2},x_{3}\in A$.
\end{definition}

A ternary Nambu-Poisson algebra is a non-commutative ternary Nambu-Poisson algebra $(A,\mu,\{.,.,.\})$ for which $\mu$ is commutative, then $\mu$ is commutative unless otherwise stated.\\
In a (non-commutative) ternary Nambu-Poisson algebra, the ternary bracket $\{.,.,.\}$ is called Nambu-Poisson bracket.\\
Similarly, a non-commutative $n$-ary Nambu-Poisson algebra is a triple $(A,\mu,\{.,\cdots ,.\})$ where  $(A,\{.,\cdots ,.\})$  defines an $n$-Lie algebra satisfying similar Leibniz rule with respect to $\mu$.\\
A morphism of (non-commutative) ternary Nambu-Poisson algebras is a linear map that is a morphism of the underlying ternary Nambu-Lie algebras and associative algebras.\\\\

\begin{example}
Let $C^{\infty}(\mathds{R}^{3})$ be the algebra of $C^{\infty}$ functions  on $\mathds{R}^{3}$ and $x_{1},x_{2},x_{3}$ the coordinates on $\mathds{R}^{3}$. We define the ternary brackets as in \eqref{jacobian}, then $(C^{\infty}(\mathds{R}^{3}),\{.,.,.\})$ is a ternary Nambu-Lie algebra. In addition the bracket satisfies the Leibniz rule: $\{f g,f_{2},f_{3}\}=f \{g,f_{2},f_{3}\}+\{f,f_{2},f_{3}\} g$ where $f,g,f_2,f_3\in C^{\infty}(\mathds{R}^{3})$ and  the multiplication being the pointwise multiplication  that is  $fg(x)=f(x) g(x)$. Therefore,  the algebra is a ternary Nambu-Poisson algebra.\\
This algebra was considered already in 1973 by Nambu \cite{n:generalizedmech} as a possibility of extending the Poisson bracket of standard hamiltonian mechanics to bracket of three functions defined by the Jacobian. Clearly, the Nambu bracket may be generalized further to a Nambu-Poisson allowing for an arbitrary number of entries.\\
In particular, the algebra of polynomials of variables $x_{1},x_{2},x_{3}$ with the ternary operation defined by the Jacobian function in \eqref{jacobian}, is a ternary Nambu-Poisson algebra.
\end{example}

\begin{remark}
The $n$-dimensional ternary Nambu-Lie algebra of Example 1.3 does not carry a non-commutative Nambu-Poisson algebra structure except that one given by a trivial multiplication.
\end{remark}

\section{Hom-type (non commutative) ternary Nambu-Poisson algebras}

In this section, we present various Hom-algebra structures. The main feature of Hom-algebra structures is that usual identities are
deformed by an endomorphism and when the structure map is the identity, we recover the usual algebra structure.\\

A  Hom-algebra (resp. ternary Hom-algebra) is a triple $(A,\nu,\alpha)$ consisting of a $\mathbb{K}$-vector space $A$, a
bilinear map $ \nu: A \times A \rightarrow A$ (resp. a trilinear map $ \nu: A \times A\times A \rightarrow A$) and a linear map $\alpha : A \rightarrow A$.
A binary Hom-algebra $(A,\mu, \alpha)$ is said to be multiplicative if $\alpha \circ \mu=\mu \circ \alpha ^{\otimes 2}$ and it is called commutative if $\mu=\mu^{op}$. A ternary Hom-algebra $(A,m, \alpha)$ is said to be multiplicative if $\alpha \circ m = m \circ \alpha ^{\otimes 3}$.
Classical binary (resp. ternary)  algebras are regarded as binary (resp. ternary) Hom-algebras with identity twisting map.
We will often use the abbreviation $xy$ for $\mu(x,y)$ when there is no ambiguity.
 For a linear map $\alpha: A\rightarrow A$, denote by $\alpha^{n}$ the $n$-fold composition of $n$-copies of $\alpha$, with $\alpha^{0}\equiv Id$. \\

\begin{definition} \label{def:HomAssAlg}
A  Hom-algebra $(A,\mu ,\alpha)$ is a \emph{ Hom-associative algebra} if it   satisfies the Hom-associativity condition,  that is 
\begin{center}
$\mu(\alpha(x),\mu(y,z)) = \mu(\mu(x , y), \alpha(z))$ for all $x, y, z \in A$.
\end{center}
\end{definition}
\begin{remark}
When $\alpha$ is the identity map, we recover the classical associativity condition, then usual  associative algebras.
\end{remark}

\begin{definition}\label{def:TerAlg}
A \emph{ternary Hom-Nambu algebra} is a triple $(A,\{.,.,.\}, \widetilde{\alpha})$ consisting of a $\mathbb{K}$-vector space $A$, a
ternary map $\{.,.,.\}: A\times A \times A\rightarrow A$ and a pair of linear maps $\widetilde{\alpha}=(\alpha_{1},\alpha_{2})$ where $\alpha_{1},\alpha_{2}:A
\rightarrow A$  satisfying

\begin{equation}
\begin{split}
 \{\alpha_{1}(x_{1}), \alpha_{2}(x_{2}), \{x_{3}, x_{4}, x_{5}\}\} =
 \{\{x_{1}, x_{2}, x_{3}\}, \alpha_{1}(x_{4}), \alpha_{2}(x_{5})\} +\\ \{\alpha_{1}(x_{3}), \{x_{1}, x_{2}, x_{4}\}, \alpha_{2}(x_{5})\} + \{\alpha_{1}
 (x_{3}), \alpha_{2}(x_{4}), \{x_{1}, x_{2}, x_{5}\}\}.
\end{split}
\end{equation}
\end{definition}
We call the above condition the ternary Hom-Nambu identity.\\

Generally, the $n$-ary Hom-Nambu algebras are defined by the following Hom-Nambu identity
\begin{align*}
&\{\alpha_{1}(x_{1}), ..., \alpha_{n-1}(x_{n-1}), \{x_{n}, ..., x_{2n-1}\}\} \\
&=\sum_{i=n}^{2n-1}\{\alpha_{1}(x_{n}), ..., \alpha_{i-n}(x_{i-1}),\{x_{1},..., x_{n-1}, x_{i}\},\alpha_{i-n+1}(x_{i+1})..., \alpha_{n-1}(x_{2n-1})\}
\end{align*}
for all $(x_{1},..., x_{2n-1})\in A^{2n-1} $.

\begin{remark}
A Hom-Nambu algebra is a \emph{ Hom-Nambu-Lie} algebra if the bracket is skew-symmetric.

\end{remark}

\begin{definition}
A \emph{non-commutative ternary Hom-Nambu-Poisson algebra} is a tuple $(A,\mu,\beta,\{.,.,.\}, \widetilde{\alpha})$ consisting of a vector space $A$, a ternary operation $\{.,.,.\}:A\times A\times A\rightarrow A$, a binary operation $\mu:A\times A\rightarrow A$, a pair of linear maps $\widetilde{\alpha}=(\alpha_{1},\alpha_{2})$ where $\alpha_{1},\alpha_2:A \rightarrow A$,  and a linear map $\beta:A \rightarrow A$ such that:
\begin{enumerate}
  \item $(A,\mu,\beta)$ is a binary Hom-associative algebra,
 \item $(A,\{.,.,.\},\widetilde{\alpha})$ is a ternary Hom-Nambu-Lie algebra,
 \item
 $\{\mu(x_{1},x_{2}),\alpha_{1}(x_{3}),\alpha_{2}(x_{4})\}= \mu(\beta(x_{1}),\{x_{2},x_{3},x_{4}\})+\mu(\{x_{1},x_{3},x_{4}\},\beta(x_{2}))$. 
\end{enumerate}

The third condition  is called  Hom-Leibniz identity.
\end{definition}


\begin{remark}
Notice that $\mu$ is not assumed to be  commutative. When $\mu$ is a commutative multiplication, then $(A,\mu,\beta,\{.,.,.\},\widetilde{\alpha})$ is said to be  a ternary Hom-Nambu-Poisson algebra.

We recover the classical (non-commutative) ternary Nambu-Poisson algebra when $\alpha_{1}=\alpha_{2}=\beta=Id$.

Similarly, a non-commutative $n$-ary Hom-Nambu-Poisson algebra is a tuple 

$(A,\mu,\beta,\{.,\cdots ,.\}, \widetilde{\alpha})$ where  $(A,\{.,\cdots,.\},\widetilde{\alpha})$ with $\widetilde{\alpha}=(\alpha_{1},\cdots , \alpha_{n-1})$  that defines an $n$-ary Hom-Nambu-Lie algebra satisfying similar Leibniz rule with respect to $(A,\mu,\beta )$.

In the sequel we will mainly interested in the class of non-commutative ternary Nambu-Poisson algebras where  $\alpha=\alpha_1=\alpha_2=\beta$, for which  we refer by a quadruple $(A,\mu,\{.,.,.\},\alpha)$.
\end{remark}

\begin{definition}
Let  $(A,\mu,\{.,.,.\},\alpha)$ be a (non-commutative) ternary Hom-Nambu-Poisson algebra.
It  is said to be  \emph{multiplicative} if
\begin{eqnarray*}
\alpha(\{x_{1},x_{2},x_{3}\})&=&\{\alpha(x_{1}),\alpha(x_{2}),\alpha(x_{3})\},\\
 \alpha \circ \mu&=&\mu \circ \alpha^{\otimes 2}.
\end{eqnarray*}
If in addition $\alpha$ is bijective then it is called  \emph{regular}.
\end{definition}

 Let $(A',\mu_{A'},\{.,.,.\}_{A'},\alpha_{A'})$ be another such quadruple. A \emph{weak morphism} $\varphi:A\rightarrow A'$ is a linear map such that
\begin{itemize}
  \item $\varphi \circ \{.,.,.\}=\{.,.,.\}_{A'}\circ \varphi^{\otimes 3}$,
  \item $\varphi \circ \mu=\mu_{\beta}\circ \varphi^{\otimes 2}$.
\end{itemize}
A morphism $\varphi:A\rightarrow A'$ is a weak morphism for which we have in addition $\varphi \circ \alpha=\alpha_{A'}\circ \varphi.$

\begin{definition}
Let $(A,\mu,\{.,.,.\},\alpha)$ and $(A',\mu',\{.,.,.\}',\alpha')$ be two  (non-commutative) ternary Hom-Nambu-Poisson  algebras.
A linear map $f : A\rightarrow A'$ is a \emph{morphism} of (non-commutative) ternary Hom-Nambu-Poisson algebras if it satisfies:
\begin{eqnarray}
f(\{x_{1},x_{2}, x_{3}\}) &=& \{f(x_{1}),f(x_{2}) ,f(x_{3})\}' ,\\
f \circ \mu&=&\mu'\circ f^{\otimes 2},\\
f\circ \alpha &= & \alpha' \circ f.
\end{eqnarray}
It said to be a  \emph{weak morphism} if    hold only  the two first conditions.
\end{definition}

\begin{proposition}
Let $(A_{1},\mu_{1},\{.,.,.\}_{1},\alpha_{1})$ and $(A_{2},\mu_{2},\{.,.,.\}_{2},\alpha_{2})$ be two ternary (non-commutative) Hom-Nambu-Poisson algebras.
A linear map $\phi:A_{1}\rightarrow A_{2}$ is a morphism of ternary (non-commutative) Hom-Nambu-Poisson algebras if and only if
 $\Gamma_{\phi}\subseteq A_{1}\oplus A_{2}$ is a Hom-Nambu-Poisson subalgebra of
$$(A_{1}\oplus A_{2},\mu_{A_{1}\oplus A_{2}},\{.,.,.\}_{A_{1}\oplus A_{2}},\alpha_{A_{1}\oplus A_{2}})$$
where $\Gamma_{\phi}=\{(x,\phi(x)): x\in A_1\}\subset A_{1}\oplus A_{2}.$
\end{proposition}

\begin{proof}
Let $\phi:(A_{1},\mu_{1},\{.,.,.\}_{1},\alpha_{1})\rightarrow (A_{2},\mu_{2},\{.,.,.\}_{2},\alpha_{2})$ be a morphism of ternary Hom-Nambu-Poisson
 algebras.\\
We have
\begin{align*}
\{x_{1}+\phi(x_{1}),x_{2}+\phi(x_{2}),x_{3}+\phi(x_{3})\}_{A_{1}\oplus A_{2}}&= \{x_{1},x_{2},x_{3}\}_{1}+\{\phi(x_{1}),\phi(x_{2}),\phi(x_{3})\}_{2}\\
&=\{x_{1},x_{2},x_{3}\}_{1}+ \phi \{x_{1},x_{2},x_{3}\}_{1}.
\end{align*}
Then $\Gamma_{\phi}$ is closed under the bracket $\{.,.,.\}_{A_{1}\oplus A_{2}}$.\\
We have also
\begin{align*}
(\alpha_{1}+\alpha_{2})(x_{1}+\phi(x_{1}))=\alpha_{1}(x_{1})+\alpha_{2}\circ \phi(x_{1})=\alpha_{1}(x_{1})+\phi \circ \alpha_{1}(x_{1}),
\end{align*}
which implies that $(\alpha_{1}+\alpha_{2})\Gamma_{\phi}\subseteq \Gamma_{\phi}$.\\
Conversely, if the graph $\Gamma_{\phi}\subseteq A_{1}\oplus A_{2}$ is a Hom-subalgebra of
$$(A_{1}\oplus A_{2},\mu_{A_{1}\oplus A_{2}},\{.,.,.\}_{A_{1}\oplus A_{2}},\alpha_{A_{1}\oplus A_{2}}),$$ 
then we have
$$\{\phi(x_{1}),\phi(x_{2}),\phi(x_{3})\}_{2}=\phi \{x_{1},x_{2},x_{3}\}_{1},$$
and
\begin{align*}
(\alpha_{1}+\alpha_{2})(x+\phi(x))&=\alpha_{1}(x)+\alpha_{2}\circ \phi(x)\in \Gamma_{\phi}
\\&=\alpha_{1}(x)+\phi \circ \alpha_{1}(x).
\end{align*}
Finally
\begin{align*}
\mu_{A_{1}\oplus A_{2}}(x_{1}+\phi(x_{1}),x_{2}+\phi(x_{2}))&=\mu_{1}(x_{1},x_{2})+\mu_{2}(\phi(x_{1}),\phi(x_{2}))\\
&=\mu_{1}(x_{1},x_{2})+\phi \circ \mu_{2}(x_{1},x_{2})\subseteq  \Gamma_{\phi}.
\end{align*}
Therefore $\phi$ is a morphism of ternary (non-commutative) Hom-Nambu-Poisson algebras.

\end{proof}

\section{Tensor product and direct sums}
In the following, we define a direct sum of two ternary (non-commutative) Hom-Nambu-Poisson  algebras.
\begin{theorem}
Let $(A_{1},\mu_{1},\{.,.,.\}_{1},\alpha_{1})$ and $(A_{2},\mu_{2},\{.,.,.\}_{2},\alpha_{2})$ be two ternary (non-commutative) Hom-Nambu-Poisson algebras.  Let $\mu_{A_{1}\oplus A_{2}}$ be a bilinear map  on $A_{1}\oplus A_{2}$ defined for $x_1,y_1,z_1\in A_1$ and $x_2,y_2,z_2\in A_1$ by 
$$\mu(x_{1}+x_{2},y_{1}+y_{2})=\mu_{1}(x_{1},y_{1})+\mu_{2}(x_{2},y_{2}),$$   $\{.,.,.\}_{A_{1}\oplus A_{2}}$  a trilinear map  defined by $$\{x_{1}+x_{2},y_{1}+y_{2},z_1+z_2\}_{A_{1}\oplus A_{2}}=\{x_1,y_1,z_1\}_1+\{x_2,y_2,z_2\}_2$$ and  $\alpha_{A_{1}\oplus A_{2}}$ a linear map defined by $$\alpha_{A_{1}\oplus A_{2}}(x_1+y_1)=\alpha_1(x_1)+\alpha_2(x_2).$$  Then 
$$
(A_{1}\oplus A_{2},\mu_{A_{1}\oplus A_{2}},\{.,.,.\}_{A_{1}\oplus A_{2}},\alpha_{A_{1}\oplus A_{2}})$$
is a ternary (non-commutative) Hom-Nambu-Poisson algebra.
\end{theorem}

\begin{proof}
The commutativity of $\mu_{A_{1}\oplus A_{2}}$ is obvious since $\mu_1$ and $\mu_2$ are commutative. The skew-symmetry of the bracket follows  from the
 skew-symmetry of $\{.,.,.\}_{1}$ and $\{.,.,.\}_{2}$. So it remains to check the Hom-associativity, the Hom-Nambu and the Hom-Leibniz
 identities.  For Hom-associativity identity, we have
\begin{align*}
&\mu_{A_{1}\oplus A_{2}}(\mu_{A_{1}\oplus A_{2}}(x_{1}+x'_{1},x_{2}+x'_{2}),\alpha_{A_{1}\oplus A_{2}}(x_{3}+x'_{3}))\\&=\mu_{A_{1}\oplus A_{2}}(\mu_{1}(x_{1},x_{2})+\mu_{2}(x'_{1},x'_{2}),\alpha_{1}(x_{3})+\alpha_{2}(x'_{3}))
\\&=\mu_{1}(\mu_{1}(x_{1},x_{2}),\alpha_{1}(x_{3}))+ \mu_{2}(\mu_{2}(x'_{1},x'_{2}),\alpha_{2}(x'_{3}))
\\&=\mu_{1}(\alpha_{1}(x_{1}),\mu_{1}(x_{2},x_{3}))+\mu_{2}(\alpha_{2}(x'_{1}),\mu_{2}(x'_{2},x'_{3}))
\\&=\mu_{A_{1}\oplus A_{2}}(\alpha_{1}(x_{1})+\alpha_{2}(x'_{1}),\mu_{1}(x_{2},x_{3})+\mu_{2}(x'_{2},x'_{3}))
\\&=\mu_{A_{1}\oplus A_{2}}(\alpha_{A_{1}\oplus A_{2}}(x_{1},x'_{1}), \mu_{A_{1}\oplus A_{2}}(x_{2}+x'_{2},x_{3}+x'_{3})).
\end{align*}
Now we prove the Hom-Nambu identity
\begin{align*}
&\{\alpha_{A_{1}\oplus A_{2}}(x_{1}+x'_{1}),\alpha_{A_{1}\oplus A_{2}}(x_{2}+x'_{2}),\{x_{3}+x'_{3}, x_{4}+x'_{4},x_{5}+x'_{5}\}_{A_{1}\oplus A_{2}}\}_{A_{1}\oplus A_{2}}\\
&=\{\alpha_{1}(x_{1})+\alpha_{2}(x'_{1}),\alpha_{1}(x_{2})+\alpha_{2}(x'_{2}),\{x_{3},x_{4},x_{5}\}_{1}+\{x'_{3},x'_{4},x'_{5}\}_{2}\}_{A_{1}\oplus A_{2}}\\
&=\{\alpha_{1}(x_{1}),\alpha_{1}(x_{2}),\{x_{3},x_{4},x_{5}\}_{1}\}_{1}+\{\alpha_{2}(x'_{1}),\alpha_{2}(x'_{2}),\{x'_{3},x'_{4},x'_{5}\}_{2}\}_{2}\\
&=\{\{x_{1},x_{2},x_{3}\}_{1},\alpha_{1}(x_{4}),\alpha_{1}(x_{5})\}_{1}+\{\alpha_{1}(x_{3}),\{x_{1},x_{2},x_{4}\}_{1},\alpha_{1}(x_{5})\}_{1}\\
&+\{\alpha_{1}(x_{3}),\alpha_{1}(x_{4}),\{x_{1},x_{2},x_{5}\}_{1}\}_{1}+\{\{x'_{1},x'_{2},x'_{3}\}_{2},\alpha_{2}(x'_{4}),\alpha_{2}(x'_{5})\}_{2}\\
&+\{\alpha_{2}(x'_{3}),\{x'_{1},x'_{2},x'_{4}\}_{2},\alpha_{2}(x'_{5})\}_{2}+\{\alpha_{2}(x'_{3}),\alpha_{2}(x'_{4}),\{x'_{1},x'_{2},x'_{5}\}_{2}\}_{2}\\
&=\{\{x_{1},x_{2},x_{3}\}_{1}+\{x'_{1},x'_{2},x'_{3}\}_{2},\alpha_{1}(x_{4})+\alpha_{2}(x'_{4}),\alpha_{1}(x_{5})+\alpha_{2}(x'_{5})\}_{A_{1}\oplus A_{2}}\\
&+\{\alpha_{1}(x_{3})+\alpha_{2}(x'_{3}),\{x_{1},x_{2},x_{4}\}_{1}+\{x'_{1},x'_{2},x'_{4}\}_{2},\alpha_{1}(x_{5})+\alpha_{2}(x'_{5})\}_{A_{1}\oplus A_{2}}\\
&+\{\alpha_{1}(x_{3})+\alpha_{2}(x'_{3}),\alpha_{1}(x_{3})+\alpha_{2}(x'_{3}),\{x_{1},x_{2},x_{5}\}_{1}+\{x'_{1},x'_{2},x'_{5}\}_{2}\}_{A_{1}\oplus A_{2}}\\
&=\{\{x_{1}+x'_{1},x_{2}+x'_{2},x_{3}+x'_{3}\}_{A_{1}\oplus A_{2}},\alpha_{A_{1}\oplus A_{2}}(x_{4}+x'_{4}),\alpha_{A_{1}\oplus A_{2}}(x_{5}+x'_{5})\}_{A{1}\oplus A_{2}}\\
&+\{\alpha_{A_{1}\oplus A_{2}}(x_{3}+x'_{3}),\{x_{1}+x'_{1},x_{2}+x'_{2},x_{4}+x'_{4}\}_{A_{1}\oplus A_{2}},\alpha_{A_{1}\oplus A_{2}}(x_{5}+x'_{5})\}_{A{1}\oplus A_{2}}\\
&+\{\alpha_{A_{1}\oplus A_{2}}(x_{3}+x'_{3}),\alpha_{A_{1}\oplus A_{2}}(x_{4}+x'_{4}),\{x_{1}+x'_{1},x_{2}+x'_{2},x_{5}+x'_{5}\}_{A_{1}\oplus A_{2}}\}_{A_{1}\oplus A_{2}}.
\end{align*}
Finally, for   Hom-Leibniz identity we have
\begin{align*}
&\{\mu_{A_{1}\oplus A_{2}}(x_{1}+x'_{1}),\alpha_{A_{1}\oplus A_{2}}(x_{3},x'_{3}),\alpha_{A_{1}+A_{2}}(x_{4},x'_{4})\}_{A_{1}\oplus A_{2}}\\
&=\{\mu_{1}(x_{1},x_{2})+\mu_{2}(x'_{1},x'_{2}),\alpha_{1}(x_{3})+\alpha_{2}(x'_{3}),\alpha_{1}(x_{4})+\alpha_{2}(x'_{4})\}_{A_{1}\oplus A_{2}}\\
&=\{\mu_{1}(x_{1},x_{2}),\alpha_{1}(x_{3}),\alpha_{1}(x_{4})\}_{1}+\{\mu_{2}(x'_{1},x'_{2}),\alpha_{2}(x'_{3}),\alpha_{2}(x'_{4})\}_{2}\\
&=\mu_{1}(\alpha_{1}(x_{1}),\{x_{2},x_{3},x_{4}\}_{1})+\mu_{1}(\{x_{1},x_{3},x_{4}\}_{1},\alpha_{1}(x_{2}))\\
&+\mu_{2}(\alpha_{2}(x'_{1}),\{x'_{2},x'_{3},x'_{4}\}_{2})+\mu_{2}(\{x'_{1},x'_{3},x'_{4}\}_{2},\alpha_{2}(x'_{2}))\\
&=\mu_{A_{1}\oplus A_{2}}(\alpha_{A_{1}\oplus A_{2}}(x_{1},x'_{1}),\{x_{2}+x'_{2},x_{3}+x'_{3},x_{4}+x'_{4}\}_{A_{1}\oplus A_{2}})\\
&+\mu_{A_{1}\oplus A_{2}}(\{x_{1}+x'_{1},x_{3}+x'_{3},x_{4}+x'_{4}\}_{A_{1}\oplus A_{2}},\alpha_{A_{1}\oplus A_{2}}(x_{2},x'_{2})).
\end{align*}
That ends the proof.
\end{proof}

Now, we define the tensor product of two ternary Hom-algebras. Moreover, we consider a tensor product of a ternary Hom-Nambu-Poisson algebra  and a  totally Hom-associative symmetric ternary algebra. \\
Let $A_{1}=(A,m,\alpha)$, where $\alpha=(\alpha_{i})_{i=1,2}$ and $A_{2}=(A',m',\alpha')$ where $\alpha'=(\alpha'_{i})_{i=1,2}$ be two
ternary (non-commutative) Hom-algebras of a given type, the tensor product $A_{1}\otimes A_{2}$ is a ternary algebra defined by the triple
$(A\otimes A', m\otimes m',\alpha\otimes \alpha')$
where $\alpha\otimes \alpha'=(\alpha_{i}\otimes \alpha'_{i})_{i=1,2}$ with
\begin{align}
m\otimes m'(x_{1}\otimes x'_{1},x_{2}\otimes x'_{2},x_{3}\otimes x'_{3})=m(x_{1},x_{2},x_{3})\otimes m'(x'_{1},x'_{2},x'_{3}),
\end{align}
\begin{align}
\alpha_{i}\otimes \alpha'_{i}=\alpha_{i}(x_1)\otimes \alpha'_{i}(x'_1),
\end{align}
where $x_{1},x_{2},x_{3} \in A_{1}$ and $x'_{1},x'_{2},x'_{3}\in A_{2}$.

Recall that $(A,m,\alpha)$ is a totally Hom-associative ternary algebra  if 
\begin{eqnarray*}
m(\alpha_{1}(x_{1}),\alpha_{2}(x_{2}),m(x_{3},x_{4},x_{5}))= m(\alpha_{1}(x_{1}),m(x_{2},x_{3},x_{4}),\alpha_{2}(x_{5}))\\
= m(m(x_{1},x_{2},x_{3}),\alpha_{1}(x_{4}),\alpha_{2}(x_{5})).
\end{eqnarray*}
for all  $x_{1}\cdots,x_{5}\in A$, 
and the ternary multiplication $m$ is symmetric if \begin{align}
m(x_{\sigma(1)},x_{\sigma(2)},x_{\sigma(3)})=m(x_{1},x_{2},x_{3}).
\end{align}
for all   $\sigma\in S_{3}$, $x_{1},x_{2},x_{3}\in A$.

\begin{lemma}
Let $A_{1}=(A,m,\alpha)$ and $A_{2}=(A',m',\alpha')$ be two ternary Hom-algebras of   given type (Hom-Nambu, totally Hom-associative). If $m$ is symmetric and $m'$ is skew-symmetric then
$m \otimes m'$ is skew-symmetric.
\end{lemma}
\begin{proof}
Straightforward. 
\end{proof}
\begin{theorem}
Let $(A,\mu,\beta,\{.,.,.\},(\alpha_{1},\alpha_{2}))$ be a ternary (non-commutative) Hom-Nambu-Poisson algebra, $(B,\tau,(\alpha'_{1},\alpha'_{2}))$ be 
a  totally Hom-associative symmetric ternary algebra, and $(B,\mu',\beta')$ be a Hom-associative algebra, then
$$(A\otimes B, \mu \otimes \mu',\beta \otimes\beta',\{.,.,.\}_{A\otimes B},(\alpha_{1}\otimes\alpha'_{1},\alpha_{2}\otimes\alpha'_{2}))$$
is a (non-commutative) ternary Hom-Nambu-Poisson algebra if and only if
\begin{align}\label{LeibAss}
\tau(\mu'(b_{1},b_{2}),b_{3},b_{4})=\mu'(b_{1},\tau(b_{2},b_{3},b_{4}))=\mu'(\tau(b_{1},b_{3},b_{4}),b_{2}).
\end{align}
\end{theorem}
\begin{proof}
Since $\mu$ and $\mu'$ are both Hom-associative multiplication whence a tensor product $\mu \otimes \mu'$ is Hom-associative. Also the commutativity of
$\mu \otimes \mu'$, the skew-symmetry of $\{.,.,.\}$ and the symmetry of $\tau$ imply the skew-symmetry of $\{.,.,.\}_{A\otimes B}$. Therefore, it  remains
to check Nambu identity and Leibniz identity.

We have
\begin{align*}
 LHS=&\{\alpha_{1}\otimes\alpha'_{1}(a_{1}\otimes b_{1}),\alpha_{2}\otimes\alpha'_{2}(a_{2}\otimes b_{2}),\{a_{3}\otimes b_{3},a_{4}\otimes b_{4},a_{5}\otimes b_{5}\}_{A\otimes B}\}_{A\otimes B}\\
&=\{\alpha_{1}(a_{1})\otimes \alpha'_{1}(b_{1}),\alpha_{2}(a_{2})\otimes \alpha'_{2}(b_{2}),\{a_{3},a_{4},a_{5}\}_{A}\otimes \tau(b_{3},b_{4},b_{5})\}_{A\otimes B}\\
&=\underbrace{\{\alpha_{1}(a_{1}),\alpha_{2}(a_{2}),\{a_{3},a_{4},a_{5}\}\}}_{a}\otimes \underbrace{\tau(\alpha'_{1}(b_{1}),\alpha'_{2}(b_{2}),\tau(b_{3},b_{4},b_{5}))}_{b},
\end{align*}
and
\begin{align*}
 RHS=&\{\{a_{1}\otimes b_{1},a_{2}\otimes b_{2},a_{3}\otimes b_{3}\}_{A\otimes B},\alpha_{1}\otimes\alpha'_{1}(a_{4}\otimes b_{4}),\alpha_{2}\otimes\alpha'_{2}(a_{5}\otimes b_{5})\}_{A\otimes B}\\
&+\{\alpha_{1}\otimes\alpha'_{1}(a_{3}\otimes b_{3}),\{a_{1}\otimes b_{1},a_{2}\otimes b_{2},a_{4}\otimes b_{4}\}_{A\otimes B},\alpha_{2}\otimes \alpha'_{2}(a_{5}\otimes b_{5})\}_{A\otimes B}\\
&+\{\alpha_{1}\otimes\alpha'_{1}(a_{3}\otimes b_{3}),\alpha_{2}\otimes\alpha'_{2}(a_{4}\otimes b_{4}),\{a_{1}\otimes b_{1},a_{2}\otimes b_{2},a_{5}\otimes b_{5}\}_{A\otimes B}\}_{A\otimes B}\\
&=\{\{a_{1},a_{2},a_{3}\}_{A}\otimes \tau(b_{1},b_{2},b_{3}),\alpha_{1}(a_{4})\otimes \alpha'_{1}(b_{4}),\alpha_{2}(a_{5})\otimes \alpha'_{2}(b_{5})\}_{A\otimes B}\\
&+\{\alpha_{1}(a_{3})\otimes \alpha'_{1}(b_{3}),\{a_{1},a_{2},a_{4}\}_{A}\otimes \tau(b_{1},b_{2},b_{4}),\alpha_{2}(a_{5})\otimes \alpha'_{2}(b_{5})\}_{A\otimes B}\\
&+\{\alpha_{1}(a_{3})\otimes \alpha'_{1}(b_{3}),\alpha_{2}(a_{4})\otimes \alpha'_{2}(b_{4}),\{a_{1},a_{2},a_{5}\}_{A}\otimes \tau(b_{1},b_{2},b_{5})\}_{A\otimes B}\\
&=\underbrace{\{\{a_{1},a_{2},a_{3}\},\alpha_{1}(a_{4}),\alpha_{2}(a_{5})\}}_{c}\otimes \underbrace{\tau(\tau(b_{1},b_{2},b_{3}),\alpha'_{1}(b_{4}),\alpha'_{2}(b_{5}))}_{d}\\
&+\underbrace{\{\alpha_{1}(a_{3}),\{a_{1},a_{2},a_{4}\},\alpha_{2}(a_{5})\}}_{e}\otimes \underbrace{\tau(\alpha'_{1}(b_{3}),\tau(b_{1},b_{2},b_{4}),\alpha'_{2}(b_{5}))}_{f}\\
&+\underbrace{\{\alpha_{1}(a_{3}),\alpha_{2}(a_{4}),\{a_{1},a_{2},a_{5}\}\}}_{g}\otimes \underbrace{\tau(\alpha'_{1}(b_{3}),\alpha'_{2}(b_{4}),\tau(b_{1},b_{2},b_{5}))}_{h}\\
\end{align*}
Using ternary Nambu identity of $\{.,.,.\}$ we have $a=c+e+g$, 
and $b=d=f=h$ using the symmetry of $\tau$ and Hom-associativity of $\mu'$, then the left hand side is equal to the right hand side from where the ternary
Hom-Nambu identity of bracket $\{.,.,.\}_{A\otimes B}$ is verified.

For the Hom-Leibniz identity, we have
\begin{align*}
LHS=&\{\mu \otimes\mu'(a_{1}\otimes b_{1},a_{2}\otimes b_{2}),\alpha_{1}\otimes\alpha'_{1}(a_{3}\otimes b_{3}),\alpha_{2}\otimes\alpha'_{2}(a_{4}\otimes b_{4})\}_{A\otimes B}\\
&=\{\mu(a_{1},b_{1})\otimes \mu'(a_{2},b_{2}),\alpha_{1}(a_{3})\otimes\alpha'_{1}(b_{3}),\alpha_{2}(a_{4})\otimes\alpha'_{2}(b_{4})\}_{A\otimes B}\\
&=\underbrace{\{\mu(a_{1},b_{1}),\alpha_{1}(a_{3}),\alpha_{2}(a_{4})\}_{A}}_{a'}\otimes \underbrace{\tau(\mu'(a_{2},b_{2}),\alpha'_{1}(b_{3}),\alpha'_{2}(b_{4}))}_{b'}\\
\end{align*}
and 
\begin{align*}
RHS=&\mu \otimes\mu'(\beta \otimes \beta'(a_{1}\otimes b_{1}),\{a_{2}\otimes b_{2},a_{3}\otimes b_{3},a_{4}\otimes b_{4}\}_{{A\otimes B}})\\
&+\mu \otimes\mu'(\{a_{1}\otimes b_{1},a_{3}\otimes b_{3},a_{4}\otimes b_{4}\}_{A\otimes B},\beta \otimes \beta'(a_{2}\otimes b_{2}))\\
&=\mu \otimes\mu'(\beta (a_{1})\otimes \beta'(b_{1}),\{a_{2},a_{3},a_{4}\}\otimes \tau(b_{2},b_{3},b_{4}))\\
&+\mu \otimes\mu'(\{a_{1},a_{3},a_{4}\}\otimes \tau(b_{1},b_{3},b_{4}),\beta (a_{2})\otimes \beta'(b_{2}))\\
&=\underbrace{\mu(\beta (a_{1}),\{a_{2},a_{3},a_{4}\})}_{c'}\otimes \underbrace{\mu'(\beta'(b_{1}),\tau(b_{2},b_{3},b_{4}))}_{d'}\\
&+\underbrace{\mu(\{a_{1},a_{3},a_{4}\},\beta (a_{2}))}_{e'}\otimes \underbrace{\mu'(\tau(b_{1},b_{3},b_{4}),\beta'(b_{2})}_{f'}
\end{align*}

With Hom-Leibniz identity we have $a'=c'+e'$, and using condition \eqref{LeibAss}
we have $b'=d'=f'$, for that the left hand side is equal to the right hand side and the Hom-Leibniz identity  is proved. Then
\begin{center}
$(A\otimes B, \mu \otimes \mu',\beta \otimes\beta',\{.,.,.\}_{A\otimes B},(\alpha_{1}\otimes\alpha'_{1},\alpha_{2}\otimes\alpha'_{2}))$
\end{center}

is a (non-commutative) ternary Hom-Nambu-Poisson algebra.

\end{proof}

\section{Construction of ternary Hom-Nambu-Poisson algebras}
In this section, we provide constructions of ternary Hom-Nambu-Poisson algebras using 
 twisting principle.

\begin{theorem}
Let $(A,\mu,\{.,.,.\},\alpha)$ be a (non-commutative) ternary Hom-Nambu-Poisson algebra and $\beta:A \rightarrow A$ be a weak morphism, then
$A_{\beta}=(A,\{.,.,.\}_{\beta}=\beta \circ \{.,.,.\},\mu_{\beta}=\beta \circ \mu,\beta \alpha)$
is also a ternary (non-commutative) Hom-Nambu-Poisson algebra. Morever, if $A$ is multiplicative and $\beta$ is a algebra morphism, then $A_{\beta}$ is
a multiplicative (non-commutative) Hom-Nambu-Poisson algebra.
\end{theorem}
\begin{proof}
If $\mu$ is commutative, then clearly so is $\mu_{\beta}$. The rest of the proof applies whether $\mu$ is commutative or not. The skew-symmetry
 follows from the skew-symmetry of the bracket $\{.,.,.\}$. It remains to prove Hom-associativity condition, Hom-Nambu-identity and Hom-Leibniz
identity. Indeed
\begin{align*}
\mu_{\beta}(\mu_{\beta}(x,y),\beta\alpha(z))&=\mu_{\beta}(\beta(\mu(x,y),\beta\alpha(z)))=\beta^{2}(\mu(\mu(x,y),\alpha(z)))\\&=\beta^{2}(\mu(\alpha(x),\mu(y,z)))=\mu_{\beta}(\beta\alpha(x),\mu_{\beta}(y,z)).
\end{align*}

We check the Hom-Nambu identity
\begin{align*}
&\{\beta\alpha(x_{1}),\beta\alpha(x_{2}),\{x_{3},x_{4},x_{5}\}_{\beta}\}_{\beta}=\beta^{2}\{\alpha(x_{1}),\alpha(x_{2}),\{x_{3},x_{4},x_{5}\}\}\\
&=\beta^{2}(\{\{x_{1},x_{2},x_{3}\},\alpha(x_{4}),\alpha(x_{5})\}
+\{\alpha(x_{3}),\{x_{1},x_{2},x_{4}\},\alpha(x_{5})\} \\ & \quad
+\{\alpha(x_{3}),\alpha(x_{4}),\{x_{1},x_{2},x_{5}\}\})\\
&=\{\{x_{1},x_{2},x_{3}\}_{\beta},\beta\alpha(x_{4}),\beta\alpha(x_{5})\}_{\beta}
+\{\beta\alpha(x_{3}),\{x_{1},x_{2},x_{4}\}_{\beta},\beta\alpha(x_{5})\}_{\beta} \\ & \quad
+\{\beta\alpha(x_{3}),\beta\alpha(x_{4}),\{x_{1},x_{2},x_{5}\}_{\beta}\}_{\beta}.
\end{align*}

Then, it remains to show Hom-Leibniz identity
\begin{align*}
&\{\mu_{\beta}(x_{1},x_{2}),\beta\alpha(x_{3}),\beta\alpha(x_{4})\}_{\beta}=\beta^{2}(\{\mu(x_{1},x_{2}),\alpha(x_{3}),\alpha(x_{4})\})\\
&=\beta^{2}(\mu(\alpha(x_{1}),\{x_{2},x_{3},x_{4}\})+\mu(\{x_{1},x_{3},x_{4}\},\alpha(x_{2})))\\
&=\mu_{\beta}(\beta\alpha(x_{1}),\{x_{2},x_{3},x_{4}\}_{\beta})+\mu_{\beta}(\{x_{1},x_{3},x_{4}\}_{\beta},\beta\alpha(x_{2})).
\end{align*}

Therefore $A_{\beta}=(A,\{.,.,.\}_{\beta},\mu_{\beta},\beta \alpha)$ is a ternary (non-commutative) Hom-Nambu-Poisson algebra.
For the multiplicativity assertion, suppose that $A$ is multiplicative and $\beta$ is an algebra  morphism.
We have
\begin{align*}
(\beta\alpha)\circ(\mu_{\beta})=\beta\alpha \circ \beta \circ \mu=\mu_{\beta}\circ \alpha^{\otimes 2} \beta^{\otimes 2}=\mu_{\beta}\circ(\beta\alpha)^{\otimes 2},
\end{align*}
and
\begin{align*}
\beta\alpha\circ\{.,.,.\}_{\beta}=\beta\alpha \circ \beta \circ \{.,.,.\}=\{.,.,.\}_{\beta}\circ(\beta\alpha)^{\otimes 3}.
\end{align*}
Then $A_{\beta}$ is multiplicative.

\end{proof}

\begin{corollary}
Let $(A,\mu,\{.,.,.\},\alpha)$ be a multiplicative ternary (non-commutative) Hom-Nambu-Poisson algebra. Then
\begin{center}
$A^{n}=(A,\mu^{(n)}=\alpha^{n} \circ \mu,\{.,.,.\}^{(n)}=\alpha^{(n)}\circ \{.,.,.\},\alpha^{n+1})$
\end{center}
is a multiplicative (non-commutative) ternary Hom-Nambu-Poisson algebra for each integer $n\geq 0$.
\end{corollary}
\begin{proof}
The multiplicativity of $A$ implies that $\alpha^{n}:A\rightarrow A$ is a Nambu-Poisson algebra  morphism.
By Theorem $4.2$  $A_{\alpha^{n}}=A^{n}$ is a multiplicative ternary (non-commutative) Hom-Nambu-Poisson algebra.
\end{proof}

\begin{corollary}
Let $(A,\mu,\{.,.,.\})$ be a ternary (non-commutative) Nambu-Poisson algebra and $\beta:A\rightarrow A$ be a Nambu-Poisson  algebra  morphism. Then
\begin{center}
$A_{\beta}=(A,\mu_{\beta}=\beta \circ \mu,\{.,.,.\}_{\beta}=\beta \circ \{.,.,.\},\beta)$
\end{center}
is a multiplicative (non-commutative) ternary Hom-Nambu-Poisson algebra.
\end{corollary}

\begin{remark}
Let $(A,\mu,\{.,.,.\},\alpha)$ and  $(A',\mu',\{.,.,.\}',\alpha')$ be two (non-commutative) ternary Nambu-Poisson algebras and $\beta:A\rightarrow A$, $\beta':A'\rightarrow A'$ be   ternary Nambu-Poisson
 algebra endomorphisms.  If $\varphi : A\rightarrow A'$ is a  ternary Nambu-Poisson algebra morphism that satisfies
$\varphi \circ \beta = \beta' \circ \varphi$, then
\begin{center}
$\varphi : (A,\mu_\beta,\{.,.,.\}_\beta,\beta\alpha) \rightarrow (A',\mu'_{\beta'},\{.,.,.\}'_{\beta'},\beta'\alpha')$
\end{center}
is a (non-commutative) ternary Hom-Nambu-Poisson algebra morphism.

Indeed, we have
\begin{center}
$\varphi \circ \{.,.,.\}_{\beta}=\varphi \circ \beta \circ \{.,.,.\}=\beta' \circ \varphi \circ \{.,.,.\}=\beta' \circ \{.,.,.\}'\circ \varphi^{\times 3}=\{.,.,.\}'_{\beta'} \circ \varphi^{\times 3}$

\end{center}
and
\begin{center}
$\varphi\circ\mu_{\beta}=\varphi\circ\beta\circ\mu=\beta'\circ\varphi\circ\mu=\beta'\circ\mu'\circ\varphi^{\times 2}=\mu'_{\beta'}\circ\varphi^{\times 2}$.
\end{center}
\end{remark}

In the sequel, we aim to construct Hom-type version of the ternary Nambu-Poisson algebra of polynomials of three variables $(\mathds{R}[x,y,z],\cdot,\{.,.,.\})$, defined in Example 1.5. The Poisson bracket of
 three polynomials is defined in \eqref{jacobian}.\\
The twisted version is given by a structure of ternary Hom-Nambu-Poisson algebra
 $(\mathds{R}[x,y,z],\cdot_{\alpha}=\alpha \circ \cdot,\{.,.,.\}_{\alpha}=\alpha \circ\{.,.,.\},\alpha)$ where $\alpha :\mathds{R}[x,y,z]\rightarrow \mathds{R}[x,y,z]$ is an algebra morphism  satisfying for all $f,g \in \mathds{R}[x,y,z]$
\begin{center}
$\alpha(f\cdot g)=\alpha(f)\cdot \alpha(g)$\\
$\alpha\{f,g,h\}=\{\alpha(f),\alpha(g),\alpha(h)\}.$
\begin{theorem}
A morphism $\alpha:\mathds{R}[x,y,z]\rightarrow\mathds{R}[x,y,z]$ which gives a structure of ternary Hom-Nambu-Poisson algebra
 $(\mathds{R}[x,y,z],\cdot_{\alpha}=\alpha \circ \cdot,\{.,.,.\}_{\alpha}=\alpha \circ\{.,.,.\},\alpha)$  satisfies the following equation:\\
\begin{align}
1-\left|
\begin{array}{ccc}
\frac{\partial  \alpha(x)}{\partial  x} & \frac{\partial  \alpha(x)}{\partial  y} & \frac{\partial  \alpha(x)}{\partial  z}\\
\frac{\partial  \alpha(y)}{\partial  x} & \frac{\partial  \alpha(y)}{\partial  y} & \frac{\partial  \alpha(y)}{\partial  z}\\
\frac{\partial  \alpha(z)}{\partial  x} & \frac{\partial  \alpha(z)}{\partial  y} & \frac{\partial  \alpha(z)}{\partial  z}
\end{array}
\right|=0,
\end{align}
\end{theorem}~\\
\begin{proof}
let  $\alpha$ be a Nambu-Poisson algebra morphism, then it satisfies for all $f, g\in \mathds{R}[x,y,z]$
\begin{center}
$\alpha(f\cdot g)=\alpha(f)\cdot \alpha(g),$\\
$\alpha\{f,g,h\}=\{\alpha(f),\alpha(g),\alpha(h)\}.$
\end{center}
The first equality shows that it is sufficient to just set $\alpha$ on $x$, $y$ and $z$.
For the second equality, we suppose by linearity that
\begin{center}
$f(x,y,z)=x^{i} y^{j}z^{k},$\\
$g(x,y,z)=x^{l} y^{m}z^{p},$\\
$f(x,y,z)=x^{q} y^{r}z^{s}$.
\end{center}
Then we can write the second equation as follows
\begin{center}
$\alpha \left|
\begin{array}{ccc}
\frac{\partial  f}{\partial  x} & \frac{\partial  f}{\partial  y} & \frac{\partial  f}{\partial  z}\\
\frac{\partial  g}{\partial  x} & \frac{\partial  g}{\partial  y} & \frac{\partial  g}{\partial  z}\\
\frac{\partial  h}{\partial  x} & \frac{\partial  h}{\partial  y} & \frac{\partial  h}{\partial  z}
\end{array}
\right|=\left|
\begin{array}{ccc}
\frac{\partial  \alpha(f)}{\partial  x} & \frac{\partial  \alpha(f)}{\partial  y} & \frac{\partial  \alpha(f)}{\partial  z}\\
\frac{\partial  \alpha(g)}{\partial  x} & \frac{\partial  \alpha(g)}{\partial  y} & \frac{\partial  \alpha(g)}{\partial  z}\\
\frac{\partial  \alpha(h)}{\partial  x} & \frac{\partial  \alpha(h)}{\partial  y} & \frac{\partial  \alpha(h)}{\partial  z}
\end{array}
\right|,$
\end{center}
which can be simplified to
\begin{align}
1=\left|
\begin{array}{ccc}
\frac{\partial  \alpha(x)}{\partial  x} & \frac{\partial  \alpha(x)}{\partial  y} & \frac{\partial  \alpha(x)}{\partial  z}\\
\frac{\partial  \alpha(y)}{\partial  x} & \frac{\partial  \alpha(y)}{\partial  y} & \frac{\partial  \alpha(y)}{\partial  z}\\
\frac{\partial  \alpha(z)}{\partial  x} & \frac{\partial  \alpha(z)}{\partial  y} & \frac{\partial  \alpha(z)}{\partial  z}
\end{array}
\right|.
\end{align}
\end{proof}
\end{center}
\begin{example}
We set polynomials:
\begin{center}
$\alpha(x)=P_{1}(x,y,z)=\sum \limits_{0\leq i,j,k\leq d}a_{ijk}x^{i}y^{j}z^{k}$,\\
$\alpha(y)=P_{2}(x,y,z)=\sum\limits_{0\leq i,j,k\leq d}b_{ijk}x^{i}y^{j}z^{k}$,\\
$\alpha(z)=P_{3}(x,y,z)=\sum\limits_{0\leq i,j,k\leq d}c_{ijk}x^{i}y^{j}z^{k}$,\\
\end{center}
where $P_{1}, P_{2}, P_{3}\in \mathds{R}[x,y,z] $, and d the largest degree for each variable.We assume that $a_{0}=b_{0}=c_{0}=0$.

\

\paragraph{\textbf{Case of polynomials of degree one}} 
We take

\begin{center}
$P_{1}(x,y,z)=a_{1}x+a_{2}y+a_{3}z$,\\
$P_{2}(x,y,z)=b_{1}x+b_{2}y+b_{3}z$,\\
$P_{3}(x,y,z)=c_{1}x+c_{2}y+c_{3}z.$
\end{center}

Equation $(2.5)$ becomes
\begin{align}
1-\left|
\begin{array}{ccc}
\frac{\partial  P_{1}(x,y,z)}{\partial  x} & \frac{\partial  P_{1}(x,y,z))}{\partial  y} & \frac{\partial  P_{1}(x,y,z)}{\partial  z}\\
\frac{\partial  P_{2}(x,y,z)}{\partial  x} & \frac{\partial  P_{2}(x,y,z)}{\partial  y} & \frac{\partial  P_{2}(x,y,z)}{\partial  z}\\
\frac{\partial  P_{3}(x,y,z)}{\partial  x} & \frac{\partial  P_{3}(x,y,z)}{\partial  y} & \frac{\partial  P_{3}(x,y,z)}{\partial  z}
\end{array}
\right|=0,
\end{align}
whence
\begin{align}
1-\left|
\begin{array}{ccc}
a_{1} & a_{2} & a_{3}\\
b_{1} & b_{2} & b_{3}\\
c_{1} & c_{2} & c_{3}
\end{array}
\right|=0.
\end{align}\\
The polynomials $P_{1}, P_{2}$ and $P_{3}$ are of one of this form\\
\begin{enumerate}

 \item
          $P_{1}(x,y,z)= x a_{1}+ y a_{2}+ z a_{3}$,
           $P_{2}(x,y,z)= b_{2}y-\frac{z}{a_{1} c_{2}}$,
          $P_{3}(x,y,z)= c_{2}y$.\\

  \item

          $P_{1}(x,y,z)= a_{1}x+  a_{2}y+ a_{3}z $,
           $P_{2}(x,y,z)= \frac{ 1 + a_{1} b_{3} c_{2}}{a_{1} c_{3}} y+ b_{3}z$,\\
           $P_{3}(x,y,z)= c_{2}y+ c_{3}z $.\\

\item $P_{1}(x,y,z)=a_{1}x+a_{2}y+a_{3}z$,
         $P_{2}(x,y,z)= b_{1}x+\frac{1}{a_{2} c_{1}}z$,
          $P_{3}(x,y,z)= c_{1}x$.\\

  \item $P_{1}(x,y,z)=a_{1}x+a_{2}y+a_{3}z$,
         $P_{2}(x,y,z)= \frac{-1+a_{2}b_{3}c_{1}}{a_{2}c_{3}}x+b_{3}z$,\\
          $P_{3}(x,y,z)= c_{1}x+c_{3}z$.\\

 \item $P_{1}(x,y,z)=\frac{a_{2}b_{1}c_{3}+b_{2}}{c_{3}x}+a_{2}y+a_{3}z$,
         $P_{2}(x,y,z)= b_{1}x+b_{2}y+b_{3}z$,\\
          $P_{3}(x,y,z)= c_{3}z$.\\

   \item $P_{1}(x,y,z)=\frac{1}{b_{2}c_{3}}x+a_{2}y+a_{3}z$,
         $P_{2}(x,y,z)= b_{2}y+b_{3}z$,
          $P_{3}(x,y,z)= c_{3}z$.\\

     \item $P_{1}(x,y,z)=a_{1}x+\frac{1}{b_{1}c_{3}}y+a_{3}z$,
         $P_{2}(x,y,z)= b_{1}x+b_{3}z$,
          $P_{3}(x,y,z)= c_{3}z$.\\

  \item $P_{1}(x,y,z)=a_{1}x+a_{2}y+ \frac{1}{b_{1}c_{2}}z$,
         $P_{2}(x,y,z)= b_{1}x$,
          $P_{3}(x,y,z)= c_{1}x+c_{2}y$.\\

  \item $P_{1}(x,y,z)=a_{1}x+\frac{-1}{b_{1}c_{3}+a_{3}c_{2}c_{3}}y+a_{3}z$,
         $P_{2}(x,y,z)= b_{1}x$,\\
          $P_{3}(x,y,z)= c_{1}x+c_{2}y+c_{3}z$.\\

  \item $P_{1}(x,y,z)=\frac{a_{2}b_{1}}{b_{2}}+\frac{1}{b_{2}c_{3}-b_{3}c_{2}}x+a_{2}y+a_{3}z$,
         $P_{2}(x,y,z)= b_{1}x+ b_{2}y+b_{3}z$,
          $P_{3}(x,y,z)= \frac{b_{1}c_{2}}{b_{2}}x+c_{2}y+c_{3}z$.\\

  \item $P_{1}(x,y,z)=\frac{-c_{3}+a_{2}c_{1}c_{2}}{b_{3}c_{2}^{2}}x+a_{2}y+a_{3}z$,
         $P_{2}(x,y,z)= b_{3}z$,\\
          $P_{3}(x,y,z)= c_{1}x+c_{2}y+c_{3}z$.\\

   \item $P_{1}(x,y,z)=a_{1}x+a_{2}y+\frac{1}{b_{1}c_{2}-b_{2}c_{1}}z$,
         $P_{2}(x,y,z)= b_{1}x+b_{2}y$,\\
          $P_{3}(x,y,z)= c_{1}x+c_{2}y$.\\

   \item $P_{1}(x,y,z)=\frac{1+a_{2}b_{1}c_{3}-a_{3}b_{1}c_{2}-a_{2}b_{3}c_{1}+a_{3}b_{2}c_{1}}{b_{2}c_{3}-b_{3}c_{2}}x+a_{2}y+a_{3}z$,\\
         $P_{2}(x,y,z)= b_{1}x+b_{2}y+b_{3}z$,
          $P_{3}(x,y,z)= c_{1}x+c_{2}y+c_{3}z$.\\

   \item $P_{1}(x,y,z)=a_{1}x+\frac{b_{2}}{b_{3}}(a_{3}-\frac{1}{b_{1}c_{2}-b_{2}c_{1}})y+a_{3}z$,
         $P_{2}(x,y,z)= b_{1}x+b_{2}y+b_{3}z$,
          $P_{3}(x,y,z)= c_{1}x+c_{2}y+\frac{b_{3}c_{2}}{b_{2}}z$.
\end{enumerate}

\

\paragraph{\textbf{Particular case of polynomials of degree two}}
We take one of the polynomials of degree two
\begin{center}
$P_{1}(x,y,z)=a_{1}x+a_{2}y+a_{3}z$\\
$P_{2}(x,y,z)=b_{1}x+b_{2}y+b_{3}z$\\
$P_{3}(x,y,z)=c_{1}x+c_{2}y+c_{3}z+c_{4}x^{2}$
\end{center}
The polynomials $P_{1}, P_{2}$ and $P_{3}$ are of one of this form\\
\begin{enumerate}

 \item $P_{1}(x,y,z)=\frac{a_{2}b_{1}}{b_{2}}+\frac{1}{b_{2}c_{3}-b_{3}c_{2}}x+a_{2}y+\frac{a_{2}b_{3}}{b_{2}}z$,
         $P_{2}(x,y,z)= b_{1}x+b_{2}y+b_{3}z$,\\
          $P_{3}(x,y,z)= c_{4}x^{2}+c_{1}x+c_{2}y+c_{3}z$.\\

  \item $P_{1}(x,y,z)=a_{2}x+\frac{a_{3}b_{2}}{b_{3}}y+a_{3}z$,
         $P_{2}(x,y,z)= b_{2}y+b_{3}z$,\\
          $P_{3}(x,y,z)= c_{4}x^{2}+c_{1}x+c_{2}y+\frac{\frac{1}{a_{1}}+b_{3}c_{2}}{b_{2}}z$.\\

  \item $P_{1}(x,y,z)=a_{2}x+a_{2}y+a_{3}z$,
         $P_{2}(x,y,z)= b_{2}y$,\\
          $P_{3}(x,y,z)= c_{4}x^{2}+c_{1}x+c_{2}y+\frac{1}{a_{1}b_{2}}z$.\\

\item $P_{1}(x,y,z)=(\frac{a_{2}b_{1}}{b_{3}}-\frac{1}{c_{2}b_{3}})x+a_{3}z$,
         $P_{2}(x,y,z)= b_{1}x+b_{3}z$,\\
          $P_{3}(x,y,z)= c_{4}x^{2}+c_{1}x+c_{2}y+c_{3}z$.\\

\item $P_{1}(x,y,z)=-\frac{1}{b_{3}c_{2}}x+a_{3}z$,
         $P_{2}(x,y,z)= b_{3}z$,\\
          $P_{3}(x,y,z)= c_{4}x^{2}+c_{1}x+c_{2}y+c_{3}z$.\\

\item $P_{1}(x,y,z)=a_{1}x-\frac{1}{b_{1}c_{3}}y+a_{3}z$,
         $P_{2}(x,y,z)= b_{1}x$,\\
          $P_{3}(x,y,z)= c_{4}x^{2}+c_{1}x+c_{3}z$.\\

 \item $P_{1}(x,y,z)=a_{1}x+\frac{-1}{b_{1}c_{3}}+\frac{a_{3}c_{2}}{c_{3}}y+a_{3}z$,
         $P_{2}(x,y,z)= b_{1}x$,\\
          $P_{3}(x,y,z)= c_{4}x^{2}+c_{1}x+c_{2}y+c_{3}z$.\\

          \item $P_{1}(x,y,z)=a_{1}x+a_{2}y+\frac{1}{b_{1}c_{2}}z$,
         $P_{2}(x,y,z)= b_{1}x$,\\
          $P_{3}(x,y,z)= c_{4}x^{2}+c_{1}x+c_{2}y$.\\

            \item $P_{1}(x,y,z)=a_{1}x+a_{2}y+a_{3}z$,
         $P_{2}(x,y,z)= \frac{(1+a_{2}b_{3}c_{1})}{a_{2}c_{3}}x+b_{3}z$,\\
          $P_{3}(x,y,z)= c_{1}x+c_{3}z$.
 \end{enumerate}

\end{example}

\section{classification}
In this section, we provide the classification of 3-dimensional ternary non-commutative Nambu-Poisson algebras. By straightforward calculations and using a computer algebra system we obtain the following result.
\begin{theorem}
 Every 3-dimensional ternary Nambu-Lie algebra is isomorphic to the ternary algebra defined with respect to basis $\{e_{1},e_{2},e_{3}\}$, by the skew
 symmetric bracket defined as
\begin{align*}
\{e_{1},e_{2},e_{3}\}=e_{1}
\end{align*}

 Moreover it define a 3-dimensional ternary non-commutative Nambu-Poisson algebra $(A,\{.,.,.\},\mu)$ if and only if $\mu$ is one of the following
 non-commutative associative algebra defined as

  \begin{enumerate}
    \item \begin{eqnarray*}
   \mu_{1}(e_{2},e_{1})=a e_{1}  & \mu_{1}(e_{2},e_{2})= e_{2} & \mu_{1}(e_{2},e_{3})= e_{3} \\
  \mu_{1}(e_{3},e_{1})=be_{1} & \mu_{1}(e_{3},e_{2})=be_{2} &  \mu_{1}(e_{3},e_{3})=be_{3},
  \end{eqnarray*}
where $a$, $b$ are parameters.
    \item \begin{eqnarray*}
        \mu_{2}(e_{1},e_{2})=a e_{1}&\mu_{2}(e_{1},e_{3})=be_{1}&\mu_{2}(e_{2},e_{2})=a e_{2}\\
    \mu_{2}(e_{2},e_{3})=be_{2}& \mu_{2}(e_{3},e_{2})=a e_{3}&\mu_{2}(e_{3},e_{3})=be_{3},
      \end{eqnarray*}
    where $a$, $b$ are parameters with  $a\neq 0$
    \item\begin{eqnarray*}
    \mu_{3}(e_{1},e_{3})=a e_{1}&\mu_{3}(e_{2},e_{3})=a e_{2}&\mu_{3}(e_{3},e_{3})=a e_{3}
         \end{eqnarray*}
          where $a$ is a parameter with  $a\neq 0$
  \end{enumerate}
The multiplication not mentioned are equal to zero.
\end{theorem}
\begin{remark}
The 3-dimensional ternary Nambu-Lie algebra is endowed with a commutative Nambu-Poisson algebra structure only when the multiplication is trivial.
\end{remark}
Using the twisting principle described in  Theorem 4.1, we obtain the following 3-dimensional non-commutative ternary Hom-Nambu-Poisson algebras.

\begin{proposition}
Any 3-dimensional ternary non-commutative Hom-Nambu-Poisson algebra $(A,\{.,.,.\}_{\alpha},\mu_{\alpha},\alpha)$ obtained by a twisting defined with respect
to the basis $\{e_{1},e_{2},e_{3}\}$ by the ternary bracket $\{e_{1},e_{2},e_{3}\}_{\alpha}=c e_{1}$, where $c$ is a parameter, and one of the following binary
Hom-associative algebra defined by  $\mu_{\alpha_{i}}$ and a corresponding structure map

\begin{enumerate}

\item
 \begin{align*}
& \mu_{\alpha_{1}}(e_{2},e_{1})  =  a c e_{1}, && \mu_{\alpha_{1}}(e_{3},e_{1})  =  b c e_{1}, \\
 & \mu_{\alpha_{1}}(e_{2},e_{2})  =a (d e_{1}+ e_{2}), && \mu_{\alpha_{1}}(e_{3},e_{2})=b(d e_{1}+ e_{2}),  \\
   & \mu_{\alpha_{1}}(e_{2},e_{3})  =a(h e_{1}+ ge_{2}+e_{3}),&& \mu_{\alpha_{1}}(e_{3},e_{3})=b(h e_{1}+ ge_{2}+e_{3}),
\end{align*}
with
\begin{align*}
\alpha_{1}(e_{1})=ce_{1},\alpha_{1}(e_{2})=d e_{1}+ e_{2},\alpha_{1}(e_{3})=h e_{1}+ ge_{2}+e_{3}.
\end{align*}
\item
\begin{align*}
&\mu_{\alpha_{2}}(e_{1},e_{2})=  a c e_{1}, && \mu_{\alpha_{2}}(e_{3},e_{1})=  b c e_{1}, \\
 &  \mu_{\alpha_{2}}(e_{2},e_{2})=a(d e_{1}+ e_{2}+l e_{3}), && \mu_{\alpha_{2}}(e_{3},e_{2})=b(d e_{1}+ e_{2}+l e_{3}),  \\
 &  \mu_{\alpha_{2}}(e_{2},e_{3})=a(h e_{1}+e_{3}),&& \mu_{\alpha_{2}}(e_{3},e_{3})=b(h e_{1}+e_{3}),
\end{align*}
with
\begin{align*}
\alpha_{2}(e_{1})=ce_{1},\alpha_{2}(e_{2})=d e_{1}+ e_{2}+l e_{3},\alpha_{2}(e_{3})=h e_{1}+e_{3}e_{3}.
\end{align*}

\item

\begin{align*}
&\mu_{\alpha_{3}}(e_{2},e_{1})=a c e_{1},&& \mu_{\alpha_{3}}(e_{3},e_{1}=b c e_{1}, \\
&\mu_{\alpha_{3}}(e_{2},e_{2})=a(d e_{1}+f e_{2}+\frac{a}{b}(1-f)e_{3}),&& \mu_{\alpha_{3}}(e_{3},e_{2})=b(d e_{1}+f e_{2}+\frac{a}{b}(1-f)e_{3}),\\
&\mu_{\alpha_{3}}(e_{2},e_{3})=a(h e_{1}+\frac{b}{a}(f-1) e_{2}+\frac{(b-g a)}{b}e_{3}),&& \mu_{\alpha_{3}}(e_{3},e_{3})=b(h e_{1}+\frac{b}{a}(f-1) e_{2}+\frac{(b-g a)}{b}e_{3}),
\end{align*}
with
\begin{align*}
\alpha_{3}(e_{1})=ce_{1}, \alpha_{3}(e_{2})=de_{1}+fe_{2}+\frac{a}{b}(1-f)e_{3}), \alpha_{3}(e_{3})=h e_{1}+\frac{b}{a}(f-1) e_{2}+\frac{(b-g a)}{b}e_{3}.
\end{align*}

\item
\begin{align*}
&\mu_{\alpha_{4}}(e_{1},e_{2})=a c e_{1},&&\mu_{\alpha_{4}}(e_{2},e_{3})=b(d e_{1}+ e_{2}), \\
&\mu_{\alpha_{4}}(e_{1},e_{3})=b c e_{1},&&\mu_{\alpha_{4}}(e_{3},e_{2})=a(h e_{1}+ ge_{2}+e_{3}),\\
&\mu_{\alpha_{4}}(e_{2},e_{2})=a(d e_{1}+ e_{2}),&& \mu_{\alpha_{4}}(e_{3},e_{3})=b(h e_{1}+ ge_{2}+e_{3}),\\
\end{align*}
with
\begin{align*}
\alpha_{4}(e_{1})=ce_{1}, \alpha_{4}(e_{2})=d e_{1}+ e_{2}, \alpha_{4}(e_{3})=h e_{1}+ ge_{2}+e_{3}.
\end{align*}

\item

\begin{align*}
&\mu_{\alpha_{5}}(e_{1},e_{2})=a c e_{1},&& \mu_{\alpha_{5}}(e_{2},e_{3})=b(d e_{1}+ e_{2}+l e_{3}),\\
&\mu_{\alpha_{5}}(e_{1},e_{3})=b c e_{1},&& \mu_{\alpha_{5}}(e_{3},e_{2})=a(h e_{1}+e_{3}),\\
&\mu_{\alpha_{5}}(e_{2},e_{2})=a(d e_{1}+ e_{2}+l e_{3}),&& \mu_{\alpha_{5}}(e_{3},e_{3})=b(h e_{1}+ e_{3}).
\end{align*}
with
\begin{align*}
\alpha_{5}(e_{1})=ce_{1}, \alpha_{5}(e_{2})=d e_{1}+ e_{2}+l e_{3}, \alpha_{5}(e_{3})=h e_{1}+ e_{3}.
\end{align*}

\item

\begin{align*}
&\mu_{\alpha_{6}}(e_{1},e_{2})=a c e_{1},&& \mu_{\alpha_{6}}(e_{2},e_{3})=b(d e_{1}+f e_{2}+\frac{a}{b}(1-f) e_{3}), \\
&\mu_{\alpha_{6}}(e_{1},e_{3})=b c e_{1},&&\mu_{\alpha_{6}}(e_{3},e_{2})=a(h e_{1}+\frac{-b}{a}(f-1) e_{2}+ \frac{b-ag}{b}e_{3}),\\
&\mu_{\alpha_{6}}(e_{2},e_{2})=a(d e_{1}+f e_{2}+\frac{a}{b}(1-f) e_{3}),&& \mu_{\alpha_{6}}(e_{3},e_{3})=b(h e_{1}+\frac{-b}{a}(f-1) e_{2}+ \frac{b-ag}{b}e_{3}),
\end{align*}
with
\begin{align*}
\alpha_{6}(e_{1})=ce_{1}, \alpha_{6}(e_{2})=d e_{1}+f e_{2}+\frac{a}{b}(1-f) e_{3}, \alpha_{6}(e_{3})=h e_{1}+\frac{-b}{a}(f-1) e_{2}+ \frac{b-ag}{b}e_{3}.
\end{align*}

\item

\begin{align*}
&\mu_{\alpha_{7}}(e_{1},e_{3})=a c e_{1},\\
&\mu_{\alpha_{7}}(e_{2},e_{3})=a(d e_{1}+f e_{2}+l e_{3}),\\
&\mu_{\alpha_{7}}(e_{3},e_{3})=a(h e_{1}+ ge_{2}+\frac{1+g+l}{f}e_{3}),
\end{align*}
with
\begin{align*}
\alpha_{7}(e_{1})=ce_{1}, \alpha_{7}(e_{2})=d e_{1}+f e_{2}+l e_{3}, \alpha_{7}(e_{3})=h e_{1}+ ge_{2}+\frac{1+g+l}{f}e_{3}.
\end{align*}

\item

\begin{align*}
&\mu_{\alpha_{8}}(e_{1},e_{3})=a c e_{1},\\
&\mu_{\alpha_{8}}(e_{2},e_{3})=a(d e_{1}+ e_{2}),\\
&\mu_{\alpha_{8}}(e_{3},e_{3})=a(h e_{1}+ ge_{2}+e_{3}),
\end{align*}
with
\begin{align*}
\alpha_{8}(e_{1})=ce_{1}, \alpha_{8}(e_{2})=d e_{1}+ e_{2}, \alpha_{8}(e_{3})=h e_{1}+ ge_{2}+e_{3}.
\end{align*}

\item

\begin{align*}
&\mu_{\alpha_{9}}(e_{1},e_{3})=a c e_{1},\\
&\mu_{\alpha_{9}}(e_{2},e_{3})=a(d e_{1}-\frac{1}{g}e_{3}),\\
&\mu_{\alpha_{9}}(e_{3},e_{3})=a(h e_{1}+ ge_{2}+re_{3}),
\end{align*}
with
\begin{align*}
\alpha_{9}(e_{1})=ce_{1}, \alpha_{9}(e_{2})=d e_{1}-\frac{1}{g}e_{3}, \alpha_{9}(e_{3})=h e_{1}+ ge_{2}+re_{3}.
\end{align*}
\end{enumerate}
\end{proposition}

\end{document}